\documentclass[11pt, reqno]{amsart}

\usepackage{textcmds} 
\usepackage{amsmath, amssymb, amsfonts, amstext, verbatim, amsthm, mathrsfs}
\usepackage[mathcal]{eucal}
\usepackage{microtype}
\usepackage[all]{xy}
\usepackage[modulo]{lineno}
\usepackage{aliascnt}
\usepackage{enumitem}
\usepackage{xspace}
\usepackage{amsfonts}
\usepackage{amssymb}
\usepackage[centertags]{amsmath}
\usepackage{amsthm}
\usepackage[margin=3cm]{geometry}
\usepackage{dsfont}
\usepackage{bm}
\usepackage{subfigure}
\usepackage{amsmath}
\usepackage{array}
\usepackage[all]{xy}

\usepackage{xfrac} 

\usepackage{graphics,graphicx} 

\newcommand{\sM}{\mathscr{M}}
\newcommand{\mM}{\mathcal{M}}
\newcommand{\m}{\mathrm{m}}
\newcommand{\qm}{\mathrm{qm}}

\newcommand{\PO}{\mathfrak{PO}}
\newcommand{\C}{\mathbb{C}}

\newcommand{\R}{\mathbb{R}}

\newcommand{\Z}{\mathbb{Z}}
\newcommand{\D}{\mathbb{D}}
\newcommand{\CP}{\mathbb{C}P}
\newcommand{\PN}{(\mathbb{C}P^1)^n}
\newcommand{\Pm}{(\mathbb{C}P^1)^{2m}}

\newcommand{\Ts}{\Theta^n_s}
\newcommand{\Tsm}{\Theta^{2m}_s}

\newcommand{\fL}{\mathfrak{L}}

\newcommand{\id}{\mathrm{Id}}
\newcommand{\im}{\mathrm{im}}
\newcommand{\ev}{\mathrm{ev}}

\newcommand{\hol}{\mathrm{hol}}

\newcommand{\RRe}{\mathrm{Re}}
\newcommand{\Aut}{\mathrm{Aut}}

\newcommand{\Ham}{\mathrm{Ham}}

\newcommand{\OP}{\operatorname}
\newcommand{\indx}{\operatorname{index}}
\newcommand{\del}{\partial}

\newcommand{\be}{\textbf{e}}
\newcommand{\bb}{\mathfrak{b}}
\newcommand{\fs}{\mathfrak{s}}
\newcommand{\subscript}[2]{$#1 _ #2$}

\newcommand{\PxP}{\mathbb{C}P^1 \times \mathbb{C}P^1}
\newcommand{\Blk}{\mathbb{C}P^2\# k\overline{\mathbb{C}P^2}}

\newcommand{\BlIII}{\mathbb{C}P^2\# 3\overline{\mathbb{C}P^2}}

\newcommand{\PxPBlII}{\mathbb{C}P^1 \times \mathbb{C}P^1\# 2\overline{\mathbb{C}P^2}}

\numberwithin{equation}{section}

\newtheorem{thm}{Theorem}[section]
\newtheorem*{thm*}{Theorem}
\newtheorem{prp}[thm]{Proposition}
\newtheorem{lem}[thm]{Lemma}
\newtheorem{cor}[thm]{Corollary}

\newtheorem{cnj}[thm]{Conjecture}
\theoremstyle{definition}
\newtheorem{dfn}[thm]{Definition}
\newtheorem{rmk}[thm]{Remark}
\newtheorem{qu}[thm]{Question}

\newtheorem{ass}[thm]{Assumption}

\begin{document}

\title[Continuum families of non-displaceable Lagrangian tori in $(\CP^1)^{2m}$]{Continuum families of non-displaceable Lagrangian tori in $(\CP^1)^{2m}$}

\author{Renato Vianna}
\footnote{The author was supported by the Herschel Smith postdoctoral fellowship from
the University of Cambridge.}

\begin{abstract} 
  
We construct a family of Lagrangian tori $\Theta^n_s \subset (\CP^1)^n$, $s \in
(0,1)$, where $\Theta^n_{1/2} = \Theta^n$, is the monotone twist Lagrangian
torus described in \cite{ChSch10}. We show that for $n = 2m$ and $s \ge 1/2$
these tori are non-displaceable. Then by considering $\Theta^{k_1}_{s_1} \times
\cdots \times \Theta^{k_l}_{s_l} \times (S^2_{\OP{eq}})^{n - \sum_i k_i} \subset
(\CP^1)^n$, with $s_i \in [1/2,1)$ and $k_i \in 2\Z_{>0}$, $\sum_i k_i \le n$ we
get several $l$-dimensional families of non-displaceable Lagrangian tori. We
also show that there exists partial symplectic quasi-states
$\zeta^{\bb_s}_{\be_s}$ and linearly independent homogeneous Calabi
quasimorphims $\mu^{\bb_s}_{\be_s}$ \cite{FO311b} for which $\Tsm$ are
$\zeta^{\bb_s}_{\be_s}$-superheavy and $\mu^{\bb_s}_{\be_s}$-superheavy. We also
prove a similar result for $(\BlIII, \omega_\epsilon)$, where $\{\omega_\epsilon
;0 < \epsilon < 1\}$ is a family of symplectic forms in $\BlIII$, for which $\omega_{1/2}$
is monotone.

\end{abstract}

\maketitle

\section{Introduction} 

In \cite{FO312}, Fukaya-Oh-Ohta-Ono construct a one-dimensional family of
non-displaceable Lagrangian tori in $(\CP^1)^2$. They arise as fibres of a
(informally called) semi-toric moment map \cite[Section~3]{Wu15}, where the
fibres over the interior of the semi-toric moment polytope are Lagrangian tori,
but over a special vertex of the polytope lies a Lagrangian $S^2$ (the
anti-diagonal) where the semi-toric moment map is not differentiable. 

The weighted barycentre of the semi-toric polytope was proven by
Oakley-Usher \cite{OU13} to be the Chekanov torus \cite{ChSch10} in $(\CP^1)^2$.
The other regular fibres are Hamiltonian isotopic to so called Chekanov type tori 
described in \cite[Example 3.3.1]{Au07}. In fact, the semi-toric Lagrangian 
fibration described in \cite{FO312} can be seen as a limit of almost toric 
fibrations, in which `most of the fibres' are Chekanov type tori, see
\cite[Section~6.4]{Vi16a} and \cite[Remark~3.1]{ToVi15}.

The definition of Chekanov type tori can be easily extended to higher
dimensions, see Definition \ref{def: Chekanov type tori}. In particular, we can
get analogues of the non-displaceable tori \cite{FO312}. We can show that these
tori are non-displaceable in $\Pm$.

\subsection{Results}

 \begin{thm}\label{thm: main}
  
For a positive even integer $n = 2m$, there is a continuum of non-displaceable Lagrangian tori
$\Theta^{2m}_s \subset (\CP^1)^{2m}$, $s \in [1/2, 1)$, for which $\Theta^{2m}_{1/2} =
\Theta^{2m}$ is the monotone twist Lagrangian torus described in \cite{ChSch10}. More 
precisely, for any Hamiltonian $\Psi \in \Ham ((\CP^1)^{2m})$, we have that
$| \Tsm \cap \Psi (\Tsm) | \ge 2^{2m}$.

\end{thm}

The case $n = 2$ was proven in \cite{FO312}. The case $n = 1$ is clearly 
false, since only the monotone circle is non-displaceable. 

\begin{qu}
  For $n \ge 3$ odd and $s \in [1/2, 1)$, are the tori $\Ts$ from Definition \ref{def: 
  Ts} (non)-displaceable?
\end{qu}

\newpage
An immediate consequence of the proof of Theorem \ref{thm: main} is

\begin{cor} \label{cor: ProductTori} 
For $s_i \in [1/2,1)$, and positive even integers $k_i$, $i = 1, \dots, l$, and $n \ge \sum_i k_i$,
the Lagrangian tori

\[ \Theta^{k_1}_{s_1} \times \cdots \times \Theta^{k_l}_{s_l} \times
(S^1_{\OP{eq}})^{n - \sum_i k_i} \subset (\CP^1)^n \]
are non-displaceable.
     
\end{cor}

Just by looking to the symplectic area spectrum of Maslov index 2 relative 
homology classes we can conclude:

\begin{prp} \label{prp:toridistinct1}
 The tori $\Ts$ is not symplectomorphic to $\Theta^{k_1}_{s_1} \times \cdots \times \Theta^{k_l}_{s_l} \times
(S^1_{\OP{eq}})^{n - \sum_i k_i}$, if $n > \sum_i k_i$.
\end{prp}

Consider the counts of holomorphic (for the standard complex structure in
$(\CP^1)^n$) Maslov index 2 disks with boundary in $\Ts$, respectively
$\Theta^{k_1}_{s_1} \times \cdots \times \Theta^{k_l}_{s_l}$ ($n = \sum_i k_i$),
passing through a fixed point. Among these, look at the count of disks that have
minimal area. For $s, s_i \in (1/2,1)$, this area is $a = 1-s$, respectively $1 -
s_i$ for some $i \in \{1,\dots, l\}$. It follows from Proposition \ref{prp:
Poten} that these counts of disks of smaller area are different if $l > 1$.
Moreover, we show in Proposition \ref{prp:HighMaslovArea} that higher Maslov
index holomorphic disks with boundary on $\Ts$ must have symplectic area bigger
than $a$.
Hence, one expect that in a generic family $J_t$ of almost complex structures,
where $J_0$ is the standard complex structure and $J_1$ is another regular
almost complex structure, $J_t$-holomorphic disks of positive Maslov index and
area smaller than $a$ can only appear in a ``birth-death'' phenomenon. This
should imply that the count of Maslov index 2 disks of symplectic area $a$ with
boundary in $\Ts$ is an invariant under generic choice of almost complex
structure, and hence under symplectomorphisms (in particular Hamiltonian
isotopies) acting on $\Ts$. This would allow us to prove:

\begin{cnj} \label{cnj:toridistinct}
   The tori $\Ts$ is not symplectomorphic to $\Theta^{k_1}_{s_1} \times \cdots \times 
   \Theta^{k_l}_{s_l}$, $n = \sum_i k_i$ -- unless $l=1$ and $s_1 = s$.
\end{cnj} 

A rigorous statement proving the invariance of the count of the Maslov index 2 disks of minimal
area in the above scenario and hence Conjecture \ref{cnj:toridistinct} is 
expected to appear in the forthcoming working of the author together with 
Egor Shelukhin and Dmitry Tonkonog.



Therefore we see that -- up to a formal proof of Conjecture 
\ref{cnj:toridistinct} -- the tori obtained here differ from products of copies of the tori
obtained in \cite{FO312} and copies of the equator in $\CP^1$.









The idea of the proof of Theorem \ref{thm: main} is that we are
able to find bulk deformations $\bb_s$ for which the bulk deformed Floer
Homology of $\Tsm$ (decorated with some weakly bounding cochain $\sigma$) is
non-zero. The invariance property of the bulk deformed Floer Cohomology under
the action of Hamiltonian diffeomorphisms \cite[Theorem~2.5]{FO311a},
allow us to conclude that the above Lagrangian tori are non-displaceable. 

 Based on the work of Fukaya-Oh-Ohta-Ono \cite{FO311b}, regarding spectral
 invariants with bulk deformations, quasimorphisms and Lagrangian Floer theory,
 we are able to strengthen our result and find families of homogeneous Calabi
 quasimorphisms $\mu^{\bb_s}_{\be_s}$ and partial symplectic quasi-states
 $\zeta^{\bb_s}_{\be_s}$, for which $\Tsm$ is $\mu^{\bb_s}_{\be_s}$-superheavy and
 $\zeta^{\bb_s}_{\be_s}$-superheavy. 
 
 For the definition of homogeneous Calabi quasimorphisms, partial symplectic quasi-states and
 the notion of superheaviness we refer the reader to \cite{EP03, EP06,
 FO311b}. 
 
 
 Following closely the notation of \cite[Lemma~23.3, Theorem~23.4]{FO311b}
 we summarise the above discussion as:

 \newpage
 
 \begin{thm} \label{thm: heavy}
  
For $s \in [1/2,1)$, there exists a bulk-deformation $\bb_s \in
H^2(\Pm,\Lambda_{+})$, and a weak bounding cochain $b_s \in H^1(\Tsm,
\Lambda_0)$ for which

\[ HF(\Tsm, (\bb_s, b_s); \Lambda_{0,nov}) \cong H^*(\Tsm; \Lambda_{0,nov})\]

Moreover, there are idempotents $\be_s$ in the bulk-deformed quantum-cohomology
$QH_{\bb_s}^*(\Pm; \Lambda_{0,nov})$, so that $\Tsm$ is $\mu^{\bb_s}_{\be_s}$-superheavy and
$\zeta^{\bb_s}_{\be_s}$-superheavy. Here $\mu^{\bb_s}_{\be_s}$,
$\zeta^{\bb_s}_{\be_s}$ are respectively the homogeneous Calabi quasimorphism and partial
symplectic quasi-states coming from the bulk-deformed spectral invariant
associated with $\be_s$ \cite[Section~14]{FO311b}.
 
 \end{thm}
 
Here $\Lambda$, $\Lambda_0$, $\Lambda_{nov}$, $\Lambda_{0,nov}$ and $\Lambda_{+}$ are the Novikov rings:

\[\Lambda = \left\{ \sum_{i \ge 0} a_iT^{\lambda_i} | \, \, a_i \in \C, \lambda_i \in
\R, \lambda_i \le \lambda_{i +1}, \lim_{i \to \infty} \lambda_i = \infty \right\}, \]

\[\Lambda_0 = \left\{ \sum_{i \ge 0} a_iT^{\lambda_i} | \, \, a_i \in \C, \lambda_i \in
\R_{\ge 0}, \lambda_i \le \lambda_{i +1}, \lim_{i \to \infty} \lambda_i = \infty \right\}, \]

\[\Lambda_{nov} = \left\{ \sum_{i \ge 0} a_i q^{n_i} T^{\lambda_i} | \, \, n_i \in \Z \, \, a_i \in \C, \lambda_i \in
\R, \lambda_i \le \lambda_{i +1}, \lim_{i \to \infty} \lambda_i = \infty \right\}, \]

\[\Lambda_{0,nov} = \left\{ \sum_{i \ge 0} a_i q^{n_i} T^{\lambda_i} | \, \, n_i \in \Z \, \, a_i \in \C, \lambda_i \in
\R_{\ge 0}, \lambda_i \le \lambda_{i +1}, \lim_{i \to \infty} \lambda_i = \infty \right\}, \]

\[\Lambda_+ = \left\{ \sum_{i \ge 0} a_iT^{\lambda_i} | \, \, a_i \in \C, \lambda_i \in
\R_{> 0}, \lambda_i \le \lambda_{i +1}, \lim_{i \to \infty} \lambda_i = \infty \right\}, \]

The formal parameter $T$ is used to keep track of area of pseudo-holomorphic disks, 
while the formal parameter $q \in \Lambda_{0,nov}$ is used to keep track of the Maslov index.

The following Corollary follows immediately from \cite[Corollary~1.10]{FO311b},
see \cite[Section~19]{FO311b} for a proof.

\begin{cor} \label{cor: linear indep. quasimorph}
  The uncountable set $\{ \mu^{\bb_s}_{\be_s} \}$ of homogeneous Calabi 
  quasimorphisms is linearly independent \cite[Definition~1.9]{FO311b}.
\end{cor}

To prove linear independency of the above homogeneous Calabi 
quasimorphisms we use that the tori are disjoint, for different values of $s$.
One could ask:

\begin{qu}
  Are the tori $\Ts$ Hamiltonian displaceable from $\Theta^{k_1}_{s_1} \times \cdots \times \Theta^{k_l}_{s_l} \times
(S^1_{\OP{eq}})^{n - \sum_i k_i}$, for $s,s_i \in (1/2,1)$? 
\end{qu}

We note that by construction, these tori intersect for $s, s_i \ge 1/2$. 
See \cite{ToVi15}, for non-displaceability in the case $n = 2$, between $\Ts$ (i.e. tori from 
\cite{FO312}) $s \ge 3/2$ and the Clifford torus $S^1_{\OP{eq}} \times 
S^1_{\OP{eq}}$.

\begin{qu}
 Are the quasimorphisms arising from (particular choice of
 bulk-deformation and weak-bounding cochain for) the tori in 
Corollary \ref{cor: ProductTori} linearly independent for different
partitions $(k_1, \dots, k_l, n - \sum_i k_i)$ of $n$? 
\end{qu}

We finish our results by pointing out that the family given in \cite{FO312}
remain non-displaceable after we perform two blowups (of the same size) on the
rank zero corners of the singular fibration described in \cite{FO312}, see Figure
\ref{fig: Bl3}. This follows from applying the same ideas as Fukaya-Oh-Ohta-Ono
did for the $\PxP$ case. 

\begin{thm} \label{thm: Bl3}
  There exists a continuous family of non-displaceable Lagrangian tori $L_s^\epsilon$ in $(\BlIII, 
  \omega_\epsilon) = (\PxPBlII, \omega_\epsilon)$, where $s \in [1/2,1)$ and $\{ \omega_\epsilon | 0 < \epsilon < 
  1\}$ is a family of symplectic forms for which $(\BlIII, 
  \omega_{1/2})$ is monotone, containing a monotone Lagrangian $L_{1/2}^{1/2}$. 
  
\end{thm}

\begin{rmk} \label{rmk: Blk}
 It is shown in \cite[Section~5]{FO311a} and \cite[Section~22]{FO311b} a family
of non-displaceable Lagrangian tori in $\Blk$, $k \ge 2$, endowed with some 
\emph{non-monotone} symplectic form.  
\end{rmk}
 
Theorem \ref{thm: Bl3} follows, in the same spirit as \cite[Theorem~1.11]{FO311b} and 
Theorem \ref{thm: heavy}, from:

 \begin{thm} \label{thm: Bl3heavy}

Let $(\BlIII, \omega_\epsilon)$ and $L_s^\epsilon$ be as in Theorem \ref{thm: Bl3}.
For $s \in [1/2,1)$, there exists a bulk-deformation $\bb_s^\epsilon \in
H^2(\BlIII,\Lambda_{+})$, and a weak bounding cochain $b_s^\epsilon \in H^1(L_s^\epsilon,
\Lambda_0)$ for which

\[ HF(\BlIII, (\bb_s^\epsilon, b_s^\epsilon); \Lambda_{0,nov}) \cong H^*(L_s^\epsilon; \Lambda_{0,nov})\]

There are idempotents $\be_s^\epsilon$ in the bulk-deformed quantum-cohomology $QH(\BlIII;
\Lambda)$, so that $L_s^\epsilon$ is $\mu^{\bb_s^\epsilon}_{\be_s^\epsilon}$-superheavy and
$\zeta^{\bb_s^\epsilon}_{\be_s^\epsilon}$-superheavy, where $\mu^{\bb_s^\epsilon}_{\be_s^\epsilon}$,
$\zeta^{\bb_s^\epsilon}_{\be_s^\epsilon}$ are the homogeneous Calabi quasimorphism
and partial symplectic quasi-states coming from the bulk-deformed spectral
invariant associated with $\be_s^\epsilon$ \cite[Section~14]{FO311b}. Moreover, the
uncountable set $\{ \mu^{\bb_s^\epsilon}_{\be_s^\epsilon} \}$ of homogeneous Calabi quasimorphisms
is linearly independent.
 
 \end{thm}

The rest of the paper is organised as follows: 

In Section \ref{sec: FloerHom}, we make a quick introduction of bulk deformed 
potential and Floer cohomology for a Lagrangian $L$ satisfying Assumption \ref{ass: 
ass}. We refer the reader to \cite{FO311a,FO311b,FO312} for a complete account.
We then prove Lemma \ref{lem: H^1} and Corollary \ref{cor: PotFloerHom}, to show
that, for a Lagrangian torus $T$, critical points of the potential gives rise to 
(bulk deformed) Floer cohomology isomorphic to the usual cohomology of $T$.
We believe that \ref{lem: H^1} is known by experts on the field, but we are
not aware of it being written.

In Section \ref{sec: Regularity}, we define the notion of a pair $(X,L)$ 
consisting of a K\"{a}hler manifold $X$ and a Lagrangian submanifold $L$
being \emph{$K$-pseudohomogeneous}, for some Lie group $K$ acting
holomorphically and Hamiltonianly on $X$, leaving $L$ invariant.
We showed that if $(X,L)$ is \emph{$K$-pseudohomogeneous}, any Maslov index
$2$ holomorphic disk with boundary on $L$ such that its boundary is 
transverse to $K$-orbits, is regular. We use that to show regularity
for the Maslov index $2$ disks with boundary in $\Ts$.

In Section \ref{sec: LagTori}, we define the Lagrangian tori $\Ts$, establish
its potential function, essentially computed in \cite{Au07,Au09}, and prove
it satisfies Assumption \ref{ass: ass}, for some regular almost complex structure $J$
with the same potential function of the standard complex structure. We also prove
Proposition \ref{prp:toridistinct1} and show that holomorphic disks of Maslov index
bigger than $2$ have area bigger than $a = 1 - s$, which we use to argue why
Conjecture \ref{cnj:toridistinct} should hold.  

In Section \ref{sec: Proof}, we compute the critical points of the potential 
bulk deformed by some cocycle in $C^2(\PN, \Lambda_+)$. We show that for $n=2m$,
there are bulk deformation $\bb_s$ and a weak bounding cochain $b_s$ which is
a critical point of the potential $\PO_{\bb_s}^{\Tsm}$. It then follows from
Corollary \ref{cor: PotFloerHom} that the bulk deformed Floer cohomology 
$HF(\Tsm, (b_s,\bb_s);\Lambda)$ is isomorphic to the cohomology of the torus.
Non-displaceability then follows from \cite[Theorem~G]{FO3Book} which is also 
stated as \cite[Theorem~2.5]{FO311a}.

In Section \ref{sec: Quasimorph}, we finish the proof of Theorem \ref{thm: 
heavy}.

Finally in Section \ref{sec: Bl3}, we describe $(\BlIII, \omega_\epsilon) =
(\PxPBlII,\omega_\epsilon)$ as two blowups of capacity $\epsilon$ on two corners
of the moment polytope of $\PxP$. The Lagrangian tori $L^\epsilon_s$ on the
blowup comes from $\Theta^2_s \in \PxP$. We compute the potential for
$L^\epsilon_s$ and show the existence of critical points for some bulk
deformation. This allow us to prove Theorems \ref{thm: Bl3} and \ref{thm:
Bl3heavy}. These tori are equivalent to the fibres of the singular
fibration given by blowing up the corners of the ``semi-toric polytope''
described in \cite{FO312}, see Figure \ref{fig: Bl3}.

\subsection*{Acknowledgements} We are very grateful to Georgios Dimitroglou
Rizell, Ivan Smith, Dmitry Tonkonog and Kaoru Ono for useful discussions.

\section{Floer homology and the potential function} \label{sec: FloerHom}

Let $X$ be a symplectic manifold and $J$ a regular and compatible almost complex
structure. Let $L$ be a Lagrangian submanifold of $X$ (with a chosen spin
structure). We consider a unital canonical $A_\infty$ algebra structure
$\{\m_k\}$ on the classical cohomology $H(L;\Lambda_{0,nov})$ \cite[Section~6]{FO312},
\cite[Corollary~5.4.6, Theorem~A]{FO3Book}. The potential function
is defined from the space of weak bounding cochains $\hat{\mM}(L)$ of $L$ to
$\Lambda_0$. We refer the reader to \cite{FO311a, FO311b, FO312, FO3Book} for
the definition. 

Suppose we are given an compatible almost complex
structure $J_0$ for which $(X,L,J_0)$ satisfy:

\begin{ass} \label{ass: ass}   
Let $\beta \in \pi_2(X,L)$. Assume that:
\begin{enumerate}[label=(\subscript{A}{\arabic*})]
  \item If $\beta$ is represented by a non-constant $J_0$-holomorphic disk, then
 $\mu_L(\beta) \ge 2$, \label{ass: A1}
 \item Maslov index 2 $J_0$-holomorphic disks are regular, \label{ass: A2}
\end{enumerate}
\end{ass}
\noindent where $\mu_L$ is the Maslov index.


  
  
  
Throughout the paper we say an almost complex structure $J$ is regular 
if it satisfies assumption $(A_2)$.    

An almost complex structure satisfying Assumption \ref{ass: ass}, automatically
satisfies \cite[Condition~6.1]{FO312}, hence by
\cite[Theorem~A.1,~Theorem~A.2]{FO312} there is an embedding of
$H^1(L,\Lambda_0)$ into $\hat{\mathcal{M}}(L)$ and restricted to
$H^1(L,\Lambda_0)$ the potential function $\PO^L$ is so that

\begin{equation} \label{eq: Pot=m0}
  \m_0^b(1) = \PO^L(b)q[L],
\end{equation}
where 

\begin{equation} \label{eq: m0}
  \m_0^b(1) = \sum_{k=0}^\infty \m_k(b,\dots,b) = \underset{\mu_L(\beta) = 2}{\sum_{\beta \in \pi_2(X, L),}} 
  q^{\mu_L(\beta)/2}T^{\int_\beta \omega} \exp(b \cap \del \beta) \ev_{0*}([\mathscr{M}_1(\beta)]). 
\end{equation}

Here $[\mathscr{M}_1(\beta)]$ is the (virtual) fundamental class of the moduli space of
$J$-holomorphic disks in the class $\beta$ with 1 marked point and $\ev_0:
\sM_1(\beta) \to L$ is the evaluation map.

Using a notation closer to \cite{Au07,Au09} we define for $\beta \in 
\pi_2(X,L)$:

\begin{equation} \label{def: coord z}
  z_\beta(L,b) = T^{\int_\beta \omega} \exp(b \cap \del \beta).
\end{equation}

Letting $\eta_\beta$ be the degree of $\ev_0: \sM_1(\beta) \to L$, we can
write: 

\begin{equation} \label{eq: Pot}
  \PO^L(b) = \underset{\mu_L(\beta) = 2}{\sum_{\beta \in \pi_2(X, L),}} 
  \eta_\beta z_\beta(L,b)
\end{equation}

We want to consider the Floer cohomology of $L$ bulk-deformed by a class
$\bb = T^\rho[\fs] \in H^2(X, \Lambda_+)$ \cite{FO311a}. The potential function will
depend on the cocycle $\bb \in C^2(X, L; \Z)$, even though the Floer cohomology
doesn't. Since we use a cocycle in degree $2$ (Poincar\'e dual to a cycle of
codimension 2) the degree of the bulked deformed $A_\infty$ maps $\m_k^\bb$
\cite[(2.6)]{FO311a} is unaffected by the bulk and the bulk deformed potential is
given by:

\begin{equation} \label{eq: PotBulked}
  \PO^{L}_{\bb}(b) = \underset{\mu_L(\beta) = 2}{\sum_{\beta \in \pi_2(X, L),}} 
  \eta_\beta \exp[(\fs \cap \beta) T^\rho] z_\beta(L,b),
\end{equation}
where $b \in H^1(L,\Lambda_0)$, is a weak bounding cochain for the curved $A_\infty$
algebra $(H(L,\Lambda_{0,nov}), \{\m_k^\bb\})$, with 

\begin{equation} \label{eq: m0bulked}
  \m_0^{b,\bb}(1) = \sum_{k=0}^\infty \m_k^\bb(b,\dots,b) = \PO^{L}_{\bb}(b)q[L]. 
\end{equation}
 
The fact that $b \in H^1(L,\Lambda_0)$ is a weak bounding cochain for
$(H(L,\Lambda_0), \{\m_k^\bb\})$ implies that we can define a (not curved)
$A_\infty$ algebra $(H(L,\Lambda_{0,nov}), \{\m_k^{b,\bb}\})$, where

\begin{equation} \label{eq: mkbulked}
  \m_k^{b,\bb}(x_1, \dots, x_k) = \sum_{j=0}^\infty 
  \m_j^\bb(b,\dots,b,x_1,b,\dots,b,x_2,b,\dots,b,x_k,b,\dots,b). 
\end{equation}

In particular, 

\begin{gather}
  (\m_1^{b,\bb})^2 = 0; \label{eq: m_1 square} \\
  \m_1^{b,\bb}(\m_2^{b,\bb}(x,y)) = \pm \m_2^{b,\bb}(\m_1^{b,\bb}(x),y) \pm 
  \m_2^{b,\bb}(x,\m_1^{b,\bb}(y)). \label{eq: Leibniz}
\end{gather}

\begin{dfn} \label{dfn: bFH}
  We define the bulk deformed Floer cohomology:
 
\begin{equation}
  HF(L,(b,\bb); \Lambda_{0,nov}) = \frac{\ker(\m_1^{b,\bb})}{\im (\m_1^{b,\bb})}
\end{equation} 
  
\end{dfn}

\begin{rmk} Strengthening Assumption \ref{ass: ass} to assume regularity of
holomorphic disks with Maslov index smaller than $n-1$, one should be able to
define the Floer cohomology using the Pearl version \cite{BC09B}, and
analogously define its bulk-deformed version, which should be isomorphic to the
one in Definition \ref{dfn: bFH}. In that framework, the proof of Leibniz rule
\eqref{eq: Leibniz} follows the same ideas as \cite[Theorem~4]{Bu10}. \end{rmk}

By the work of Fukaya-Oh-Ohta-Ono, we have:

\begin{thm}[ Theorem G \cite{FO3Book}, Theorem 2.5 \cite{FO311a}] \label{thm: FOOOnonDisp}
  If $\psi: X  \to X$ is a Hamiltonian diffeomorphism, then the order of $\psi(L)\cap L$ is 
  not smaller than the rank of $HF(L,(b,\bb); \Lambda_{0,nov})\otimes_{\Lambda_{0,nov}} \Lambda_{nov}$. 
\end{thm}

We would like to point out that the product $\m_2^{b,\bb}$ can be thought as
deformation of the cup product in the sense that for $x,y \in H(L,\Lambda_0)$
of pure degrees $|x|$ and $|y|$,

\begin{equation} \label{eq: cup*}
  \m_2^{b,\bb}(x,y) = x \cup y + \text{other terms}
\end{equation}

where $x \cup y$ comes from counting constant disks and the other terms is a sum
of elements of degree smaller than $|x| + |y|$ in $H(L,\Lambda_{0,nov})$, since
it comes from evaluating moduli spaces $\mathscr{M}_{k,l+1}(\beta)$ to a cycle of
dimension $|x| + |y| - \mu_L(\beta)$ and $(X,L,J)$ satisfies Assumption \ref{ass: A1}.

The following Lemma is well established for the monotone case in \cite{Bu10}, and
in the general case in \cite{FO312}.

\begin{lem}[Theorem 2.3 of \cite{FO312}] \label{lem: H^1}
  Take $(X,L)$ satisfying Assumption \ref{ass: ass}. Also assume that
  $H(L,\Lambda_0)$ is generated by $H^1(L,\Lambda_0)$ as an algebra with respect to the
  classical cup product. If $\m_1^{b,\bb}|_{H^1(L,\Lambda_{0,nov})} = 0$ then
  $\m_1^{b,\bb} \equiv 0$.
\end{lem}

\begin{proof}

First we point out that $\m_1^{b,\bb}|_{H^0(L,\Lambda_{0,nov})} = 0$. 
Since $H(L,\Lambda_0)$ is generated by $H^1(L,\Lambda_0)$ with respect to the
cup product, we only need to show by induction on the degree that for $x$ and 
$y$ of pure degree $|x| \ge 1$, $|y|\ge 1$,
$\m_1^{b,\bb}(x \cup y) = 0$, if $\m_1^{b,\bb}(z) = 0$ for all $z$, such that
$|z| < |x| + |y|$. Using \eqref{eq: cup*},
 
\begin{equation*}
 \m_1^{b,\bb}(x \cup y) =  \m_1^{b,\bb}(\m_2^{b,\bb}(x,y))  -  \m_1^{b,\bb}(\text{other terms}) 
 = 0
\end{equation*}

by induction hypothesis and using the Leibniz rule \eqref{eq: Leibniz}. 
 
\end{proof}

\begin{rmk} \label{rmk: Buh}
  
  Lemma \ref{lem: H^1} strengthen the result of \cite[Theorem~6.4.35]{FO3Book}
  and \cite{Bu10}, showing that the minimal Maslov number $M_L$ of any Lagrangian
  torus $L$ (or any orientable Lagrangian such that the cohomology ring is
  generated by $H^1$) in $\C^n$ is 2, provided $T$ satisfies Assumption
  \ref{ass: ass} for some $J$. That is because the Lagrangian is orientable and
  $HF(T, (b,\bb); \Lambda) \equiv 0$ (from Theorem \ref{thm: FOOOnonDisp}, since
  $T$ is displaceable), so there must be a Maslov index $2$ disk. The inequality
  $2 \le M_L \le n +1$ was proven in \cite[Theorem~6.1.17]{FO3Book}, for any
  spin Lagrangian $L\subset \C^n$ satisfying Assumption \ref{ass: ass}, via the
  use of spectral sequence.
  
  \end{rmk}

\begin{dfn} \label{dfn: crit Point} Take $(X,L)$ satisfying the assumptions of
Lemma \ref{lem: H^1}. Assume that $\pi_1(L) \cong H_1(L, \Z)$ and $\pi_2(X, L)
\cong \pi_2(X) \oplus H_1(L, \Z)$. So, we are able to write the Potential
function \eqref{eq: PotBulked} in terms of $z_i = z_{\beta_i}$, for some $\beta_1, \dots,
\beta_n \in \pi_2(X, L)$, where $\del \beta_1, \dots, \del \beta_n$ is a basis
of $H_1(L, \Z)$. We say that $b$ is a \emph{critical point} of $\PO^{L}_{\bb}(b)$ 
if:
$$z_i \frac{\del
\PO^{L}_{\bb}(b)}{\del z_i} = 0.$$

\end{dfn}

\begin{cor}[Theorem 2.3 of \cite{FO312}] \label{cor: PotFloerHom}
 
Take $(X,L)$ satisfying the assumptions of Lemma \ref{lem: H^1} and Definition
\ref{dfn: crit Point}. If $b$ is a critical point of $\PO^{L}_{\bb}(b)$
\eqref{eq: PotBulked} for $\bb = T^{\rho}[\fs] \in H^2(X, \Lambda_+)$, then
$HF(L,(b,\bb);\Lambda) \cong H(L; \Lambda)$.

\end{cor}

\begin{proof}
  
  Take a basis $x_1, \dots, x_n$ a basis of $H_1(L, \Z)$. Let $\beta_1, \dots,
  \beta_n \in \pi_2(X, L) \cong \pi_2(X) \oplus H_1(L, \Z)$, be so that $\del
  \beta_i = x_i \in H_1(L, \Z)$ and write the Potential $\PO^{L}_{\bb}(b)$
  \eqref{eq: PotBulked} in terms of $z_i = z_{\beta_i}$.
  
 Since $\fs$ is of degree $2$, we have that $\m_1^{b,\bb}(\sigma)$ for $\sigma
 \in H^1(L,\Lambda)$, only counts contributions of Maslov index 2 disks. A 
 Maslov index 2 $J$-holomorphic disk in the class $\beta = \gamma + k_1 \beta_1
 + \cdots + k_n \beta_n$, $\gamma \in \pi_2(X)$ contributes to $\m_1^{b,\bb}(\sigma)$
 as
 
 $$ \sum _i k_i(\sigma \cap x_i) \eta_\beta \exp[(\fs \cap \beta)T^\rho]
 T^{\int_\gamma \omega} z_1^{k_1}\cdots z_n^{k_n} $$
 
Summing all contributions of Maslov index 2 $J$-holomorphic disks we have:

$$ \m_1^{b,\bb}(\sigma) = \sigma \cap \sum_i x_i \left(z_i \frac{\del
\PO^{L}_{\bb}(b)}{\del z_i} \right)$$
 
 Therefore, if $b$ is a critical point of $\PO_{\bb}^L(b)$, we have that
 $\m_1^{b,\bb}|_{H^1(L,\Lambda)} = 0$ and by Lemma \ref{lem: H^1},
 $\m_1^{b,\bb} \equiv 0$, so $HF(L,(b,\bb);\Lambda_{0,nov}) \cong H(L; \Lambda_{0,nov})$.
  
\end{proof}

\section{Regularity Lemma} \label{sec: Regularity}

We now move to the K\"{a}hler setting and we discuss a Lemma that we will use to
prove regularity for Maslov index 2 disks with boundary on $\Ts$ with respect to
the standard complex structure in $\PN$. The following definition is inspired in
\cite[Definition~1.1.1]{EL15b}.

\begin{dfn} \label{dfn: almHomogeneuos}
  Let $L$ be a $n$ dimensional Lagrangian in a K\"{a}hler manifold $X$. Assume 
  that $K$ is a Lie group of dimension $n-1$ acting Hamiltonianly and 
  holomorphically on $X$ preserving $L$. Assume that the action restricted to $L$
  is free. Then we say that $(X,L)$ is \emph{$K$-pseudohomogeneous}.  
\end{dfn}

We get then the following Lemma:

 \begin{lem} \label{lem: almHom}
   Let $(X,L)$ be $K$-pseudohomogeneous, for some Lie group $K$. 
   If $u$ is a Maslov index 2 holomorphic disk such that $\del 
   u$ is transverse to the $K$-orbits, then $u$ is regular.
 \end{lem}

The proof of the above Lemma relies on the Lemmas below, very similar to 
\cite[Lemmas~5.19, \ 5.20]{Vi13}.

\begin{lem} \label{lem: Vi's}
    
   Let $u: \D \to X$ be a Maslov index 2 disk in a K\"ahler manifold $X$ of
  complex dimension $n$ with boundary on a Lagrangian $L$. Assume that
$u_{|\del \D}$ is an immersion. Call $W = du(r\sfrac{\del}{\del \theta})$ a
holomorphic vector field along $u$ vanishing at $0$ and tangent to the boundary.
Assume also that there exists $V_1,\dots, V_{n-1}$ holomorphic vector fields in
$u^* TX$ such that $W \wedge V_1 \wedge \cdots \wedge V_{n-1} \ne 0$ along the
boundary of $u$. Then $u$ is an immersion and no linear combination of the
$V_i$'s is tangent to $u(\D)$.
       
   \end{lem}
   
   \begin{proof} Up to reparametrization, we may assume $du(0) \ne 0$. 
   The result follows from the fact that the zeros of $\det^2 (W \wedge V_1 \wedge
   \cdots \wedge V_{n-1})$ computes the Maslov index, which is assumed to be
   $2$. So $W \wedge V_1 \wedge \cdots \wedge V_{n-1}$ can only vanish once
   (with order 1). Since $W$ already vanishes at $0$, we cannot have either
   $du(x) = 0$ or a linear combination of the $V_i$'s being a complex multiple
   of $W$.
   \end{proof}
   
   \begin{lem} \label{lem: reg}
     
     Let $u_{\theta_1, \dots, \theta_{n-1}}$ be an $n-1$ dimensional family of
     Maslov index 2 holomorphic disks in a K\"{a}hler manifold $X$ of complex
     dimension $n$, $\theta_i \in (-\epsilon, \epsilon)$. If $u:= u_{0, \dots,
     0}$ and $V_i := \frac{\del u}{\del \theta_i}$ satisfy the hypothesis of
     Lemma \ref{lem: Vi's}, then $u$ is regular. 
     
   \end{lem}
   
   \begin{proof}

It follows similar arguments as in \cite[Lemma~5.19]{Vi13}. Using Lemma
\ref{lem: Vi's}, we are able to split $u^* TX = T\D \oplus \fL_1 \oplus \cdots
\oplus \fL_n$, as holomorphic vector bundles where $\fL_i$ is the trivial line
bundle generated by $V_i$. Also, $u_{|\del \D}^* TL = T\del \D \oplus \RRe (\fL_1)
\oplus \cdots \oplus \RRe(\fL_n)$. As in \cite[proof of Lemma~5.19]{Vi13}, we see
that the kernel of the linearised $\bar{\del}$ operator is isomorphic to      
      
  $$ T_{\mathrm{Id}} \Aut (\D) \bigoplus_{i = 1}^{n-1} \hol((\D, \del \D), (\C,\R)) $$  
  
Hence the kernel has dimension $n + 2 = n + \mu_{\Ts}(u) = \indx$.    
      
   \end{proof}

   \begin{proof}[Proof of Lemma \ref{lem: almHom}]
     
 Since the $K$ action is holomorphic and $\del u$ is transverse to the
 $K$-orbits, we can build $u_{\theta_1, \dots, \theta_n}$ from a neighbourhood
 of $\id \in K$, satisfying all the hypothesis of Lemma \ref{lem: reg}.
   
   \end{proof}

\section{The Lagrangian tori $\Theta^n_s$} \label{sec: LagTori}

In this section we give an explicit description of the tori $\Ts$ and of its
potential function, which encodes the number of Maslov index 2 disks that $\Ts$
bounds. For a definition of the potential, we refer the reader to
\cite[Section~4]{FO310},\cite{FO3Book}. See also the definition of
superpotential in \cite[Section 2.2]{Au09}.

The tori $\Ts$ appears as fibres of a singular Lagrangian fibration analogous to
the one described in ~\cite[Example 3.3.1]{Au09}.  

\subsection{Definition of $\Ts$} \label{subsec: DfnTs}

Consider $(\CP^1)^n$ with the standard symplectic form, for which the symplectic
area of each $\CP^1$ factor is $1$. For $1 \le i \le n$, let $[x_i : y_i]$ denote the
$i$-th coordinate of $(\CP^1)^n$. Consider the function $f = \prod_i
\frac{x_i}{y_i}$, defined from the complement of $V = \bigcup_{i,j} \{x_i = 0\}
\cap \{y_j = 0\}$ to $\CP^1$, whose fibres are preserved by the $T^{n-1}$ action given by

\begin{gather} (\theta_1, \dots, \theta_{n-1})\cdot ([x_1:y_1], \dots,
[x_{n-1}:y_{n-1}], [x_n:y_n]) \nonumber \\ = ([e^{\theta_1}x_1:y_1], \dots,
[e^{i\theta_{n-1}}x_{n-1},y_{n-1}], [e^{-i\sum_j \theta_j}x_n:y_n]), \label{eq: action}
\end{gather} 

and $\m: (\CP^1)^n \to \R^{n-1}$ its moment map.

\begin{dfn} \label{def: Chekanov type tori} Let $\gamma$ be an embedded circle
on $\C^{\star}$, not enclosing $0 \in \C$, and $\lambda \in \R^{n-1}$. Define
the $\Theta^n$-type Lagrangian torus:
  
   \[\Theta^n_{\gamma, \lambda} = \{ x \in (\CP^1)^n \setminus V ; f(x) \in
   \gamma, \m(x) = \lambda\} \] 
   
\end{dfn}
 
Noting that $\m^{-1}(0) = \{|x_i/y_i| = |x_n/y_n|, \forall i = 1, \dots, n-1
\}$, one can see, by using the maximum principle, that $\Theta^n_{\gamma, 0}$
bounds only one $(n-1)$-family of holomorphic disks that project injectively to
the interior of $\gamma$. Call $\beta_\gamma \in \pi_2((\CP^1)^n,
\Theta^n_{\gamma, 0})$ the class represented by each of the above disk. We note
that there are $n$ disjoint holomorphic disks in the class $\beta_\gamma$ inside
the line $\Delta = \{[x_i:y_i] = [x_n:y_n], \forall i = 1, \dots, n-1\}$. Since
$\int_{\Delta} \omega = n$, we see that $\int_{\beta_\gamma} \omega \in (0, 1)$. 

Foliate $\C \setminus \R_{\le 0}$ by curves $\gamma_s$, $s \in [0,1)$ so that
$\gamma_0$ is a point, say $1 \in \C$, and for $s \in (0,1)$, $\gamma_s$ is an
embedded circle so that $\int_{\beta_{\gamma_s}} \omega = s$.
   
\begin{dfn} \label{def: Ts}
  Define the Lagrangian torus $\Ts$ to be $\Theta^n_{\gamma_s, 0}$. 
\end{dfn}   

The hamiltonian isotopy class of $\Ts$, does not depend in the curve $\gamma_s$ 
inside $\C \setminus \R_{\le 0}$, but only on $s = \int_{\beta_{\gamma_s}} 
\omega$.

Consider the divisor $D = f^{-1}(1) \bigcup_i \{y_i =0\}$ and the holomorphic 
$n$-form $\Omega = (\prod_i x_i - 1)^{-1} dx_1\wedge \cdots \wedge dx_n$
defined on $\PN \setminus D$, in coordinates charts $y_i = 1$. 

\begin{prp}[Auroux] \label{prp: SpecialLag} 
  The tori $\Ts$ are special Lagrangians \cite[Definition~2.1]{Au07} with respect 
  to $\Omega$
\end{prp}
 
 \begin{proof}
   See \cite[Example~3.3.1]{Au09} and \cite[Proposition~5.2]{Au07}.
 \end{proof}
  
Also, we clearly have:
  
\begin{prp}

We have that $(\PN,\Ts)$ is $T^{n-1}$-pseudohomogeneous, for the action
\eqref{eq: action}. 
 
\end{prp}

\subsection{The Potential of $\Ts$} \label{subsec: PotTs}

We come back to our Lagrangian tori $\Ts$. We would like to describe the
potential $\PO^L$ in coordinates of the form \eqref{def: coord z} with
respect to a nice basis for $\pi_2(\PN, \Ts)$. Fix a point $a_s \in \gamma_s$.
Consider the $S^1$ action given by the $i$-th coordinate of the $T^{n-1}$ action
described in \eqref{eq: action}. Take the orbit lying in $\Ts \cap f^{-1}(a_s)$ and
consider its parallel transport over the segment $[0, a_s]$, formed by orbits of
the considered $S^1$ action that collapse to a point over $0$, giving rise to a
Lagrangian disk. Define $\alpha_i \in \pi_2(\Ts, \PN)$ to be the class of the
above disk. Also, from now one we write $\beta = \beta_{\gamma_s}$ and $H_i =
p_i^*[\CP^1] \in \pi_2(\PN)$ the pullback of the class of the line by the $i$-th
projection. Note that $\beta, \alpha_1, \dots, \alpha_{n-1}, H_1, \dots, H_n$
are generators of $\pi_2(\PN, \Ts)$. We assume that our monotone symplectic form
is so that $\int_{H_i} \omega = 1$.

Set $u = z_\beta$ and $w_i = z_{\alpha_i}$, $i \in (1, \dots, n-1)$. Note that
$z_{H_i}(\nabla') = T^{\int_{H_i} \omega} \exp(b \cap \del H_i) = T$.

\begin{prp}[\cite{Au07, Au09}]\label{prp: Poten} 
  
The potential function encoding the count of Maslov index 2 holomorphic
disks with boundary on the Lagrangian tori $\Theta^n_{s}$ (for some spin
structure) is given by

\begin{equation} \label{eq: PotPN} 
\PO^{\Ts} = u + \frac{T}{u}(1 + w_1 + \cdots + w_{n-1})\left(1 + \frac{1}{w_1} + \cdots + 
\frac{1}{w_{n-1}}\right)   
\end{equation}

\end{prp}

\begin{proof}[Idea of proof]
  
First we consider positivity of intersection of an holomorphic disk with the
complex submanifolds $\{x_i = 0\}$, $\{y_i = 0\}$, $\{\prod_i x_i = \prod_i
y_i\}$, for all $i \in (1, \dots, n)$, to conclude that Maslov index 2 classes
admitting holomorphic representatives must be of the form $\beta$, $H_i - \beta
- \alpha_i + \alpha_j$, where $i,j = 1, \dots, n$ and $\alpha_n = 0$.
Computations of the holomorphic disks and their algebraic count can be done
explictly. We omit here since it follows a straightforward procedure as in
~\cite[Proposition 5.12]{Au07}, see final remark after Proposition 3.3 in
\cite{Au09}. See also \cite[Section~5]{Vi13} for similar computations.  

We can choose a spin structure so that every disk counts positively, i.e.,
$\ev_0: \mathscr{M}_1 \to \Ts$ is orientation preserving, e.g. by choosing a
trivialisation of $T\Ts$ using the boundary of $\{\alpha_1,\cdots,\alpha_{n-1},
\beta\}$, as spin structure. See \cite[Section~5.5]{Vi13} and
\cite[Section~8]{Cho04}, for a complete discussion in a similar scenario. 
 
\end{proof}

\begin{rmk}
  The potential of $\Ts$ can be obtained from the known potential 
  for the Clifford torus, $\underset{n}{\times} S^1_{\OP{eq}}$. It is given by
  
\[ \PO^{\OP{Clif}} = z_1 + \cdots + z_n + \frac{T}{z_1} + \cdots + \frac{T}{z_n} .\]
  
 We obtain the potential for $\Ts$ via wall-crossing transformation $u =
 z_n(1 + w_1 + \cdots w_{n-1})$, $w_i = z_i/z_n$. See ~\cite[Example
 3.3.1]{Au09}. \end{rmk}

 \begin{prp} \label{prp: Ass}
   The tori $\Ts$ satisfy Assumption \ref{ass: ass}, with respect to the 
   standard complex structure of $\PN$.
 \end{prp}
 
 \begin{proof}

 To prove Assumption \ref{ass: A1} we use similar argument as in \cite[Example~3.3.1]{Au07}.
 First we use that $\Ts$ are special Lagrangians, and hence, by
 \cite[Lemma~3.1]{Au07}, the Maslov index is twice the intersection with the
 divisor $D$. This shows that $\mu_{\Ts}(\beta) \ge 0$, $\forall \beta \in
 \pi_2(\PN,\Ts)$ represented by an holomorphic disk $u$. Now, if $u$ is a Maslov
 index $0$ holomorphic disk, then $f \circ u$ is well define and lies in $\C
 \setminus \{1\}$, hence it is a constant in $\gamma_s$. Since the regular fibres
 of $f$ are diffeomorphic to $(\C^*)^{n-1}$, we have that $u$ is itself is constant.   
 
 The proof of Assumption \ref{ass: A2} follows from $(\PN,\Ts)$ being $T^{n-1}$-pseudohomogeneous
 together with Lemma \ref{lem: almHom}. We just need to check that since the 
 $T^{n-1}$-orbit in $\Ts$ is generated by $\del \alpha_i$, therefore transverse 
 to the boundary of the Maslov index 2 disks with boundary in $\Ts$, whose relative
 homotopy classes are $\beta$ and $H_i - \beta - \alpha_i + \alpha_j$, $i,j = 1, \dots,n$
 and $\alpha_n = 0$.

 \end{proof}

 \subsection{Regarding Proposition \ref{prp:toridistinct1}, and Conjecture \ref{cnj:toridistinct}}
 
 We start noting that Maslov index 2 classes in $H_2(\PN ,\Ts;\Z)$ are of the 
 form 
 
 \begin{equation}
   \beta + k_1 (H_1 -2\beta) + \cdots + k_n(H_n - 2\beta) + l_1\alpha_1 + \cdots + 
   l_{n-1}\alpha_{n-1},
 \end{equation}
 where $\beta$ is the Maslov index 2 and $\alpha_i$ the Maslov index $0$ classes
 described in Section \ref{subsec: PotTs}, viewed in $H_2(\PN ,\Ts;\Z)$ via
 $\pi_2(\PN, \Ts) \hookrightarrow H_2(\PN ,\Ts;\Z)$. Recalling that $\int_{H_i}
 \omega = 1$ and $\int_{\alpha_i} \omega = 0$, we see that area of Maslov index
 2 disks belongs to $\{s + (1 - 2s)\Z\} \subset \R$.
 
 \begin{proof}[Proof of Proposition \ref{prp:toridistinct1}]
   We note that each torus 
   $$\Theta^{k_1}_{s_1} \times \cdots \times \Theta^{k_l}_{s_l} \times
(S^1_{\OP{eq}})^{n - \sum_i k_i}$$ 
bounds a disk of Maslov index 2 and symplectic area $1/2$, 
if $n > \sum_i k_i$, coming from a Maslov index 2 disk in the last $\CP^1$
factor, with boundary in its equator $S^1_{\OP{eq}}$. We see that $1/2$ is in
$\{s + (1 - 2s)\Z\}$ if and only if $s=1/2$. This rules out the possibility of
$\Theta^{k_1}_{s_1} \times \cdots \times \Theta^{k_l}_{s_l} \times
(S^1_{\OP{eq}})^{n - \sum_i k_i}$ being symplectomorphic to $\Ts$ for $s \ne 1/2$.

For $s = 1/2$ the torus $\Ts$ is monotone, hence the Maslov index 2
$J$-holomorphic disks becomes an invariant of its symplectomorphism class --
this was first pointed out in \cite{ElPo93}, see also \cite[Theorem~6.4]{Vi13}.
This invariant allows us to distinguish between (the symplectomorphism classes of)
$\Ts$ and $\Theta^{k_1}_{s_1} \times \cdots \times \Theta^{k_l}_{s_l} \times
(S^1_{\OP{eq}})^{n - \sum_i k_i}$. For instance, one could look for pairs
$(\sigma_1, \sigma_2)$ of (relative homotopy classes represented by) Maslov
index 2 holomorphic disks with $\del \sigma _1 = -\del \sigma_2$. For the
torus $\Ts$, we must have $\del \sigma_i = \pm \del \beta$, i.e., only one possibility
for $\del \sigma_i$ modulo sign, see Proposition
\ref{prp: Poten}. But for each torus $\Theta^{k_1}_{s_1} \times \cdots \times
\Theta^{k_l}_{s_l} \times (S^1_{\OP{eq}})^{n - \sum_i k_i}$ we have more than 
one possibility for $\del \sigma_i$, modulo sign. 
 \end{proof}

\begin{rmk} Note that, by Proposition \ref{prp: Poten}, the total number of
Maslov index 2 holomorphic disks with boundary in $\Ts$ is $1 + n^2$, while for
the tori $\Theta^{k_1}_{s_1} \times \cdots \times
\Theta^{k_l}_{s_l} \times (S^1_{\OP{eq}})^{n - \sum_i k_i}$ it is 
$\sum_{i=1}^{l} (1 + k_i^2) + 2(n - \sum_{i=1}^{l}k_i) = 2n + \sum_{i=1}^{l} (k_i - 1)^2$.
Hence they can be equal if $(n-1)^2 = \sum_{i=1}^{l} (k_i - 1)^2$. 
\end{rmk}

 \begin{rmk} The above argument also proves the monotone version ($s = 1/2$) of
 Conjecture \ref{cnj:toridistinct}.   
\end{rmk}
 
 We proceed now to show that holomorphic disks with boundary in $\Ts$ with 
 Maslov index bigger than $2$ have area bigger than $a = 1 - s$ -- the minimal area
 of Maslov index 2 holomorphic disks for $s > 1/2$. 
 
 \begin{prp}\label{prp:HighMaslovArea}
   For $k>0$ and $s \in [1/2,1)$, the area of holomorphic Maslov index $2k$ disk with boundary on $\Ts$ is  
   least $1 - s$, with respect to the standard complex structure in $\PN$. The 
   minimum only occur if $k=1$.
 \end{prp}

\begin{proof}
  Maslov index $2k$ disks are in relative classes of the form 
  \begin{equation}\label{eq:M2k}
   k\beta + k_1 (H_1 -2\beta) + \cdots + k_n(H_n - 2\beta) + l_1\alpha_1 + \cdots + 
   l_{n-1}\alpha_{n-1}.
 \end{equation}
 
 If they are represented by holomorphic disks, their intersection with the 
 divisors $\{y_i = 0\}$ and $\{\prod_{i=1}^n x_i =  \prod_{i=1}^n y_i 
 \} = \overline{\{f^{-1}(1)\}}$ is non-negative -- recall from Definitions \ref{def: Chekanov type tori}, \ref{def: Ts}
 that $1$ is in the interior of $\gamma \subset 
 \C^*$. Noting that $$\beta \cdot \{y_i = 0\} = 0, \ \ \alpha_j \cdot \{y_i = 0\} = 0, \ \ 
 H_j \cdot \{y_i = 0\} = \delta_{ij},$$ and $$\beta \cdot \overline{\{f^{-1}(1)\}} = 1, \ \ 
 \alpha_j \cdot \overline{\{f^{-1}(1)\}} = 0, \ \  
 H_j \cdot \overline{\{f^{-1}(1)\}} = 1,$$ 
 $i,j = 1, \dots, n$, we get that
 
$$ k_i \ge 0 \ \forall i = 1, \dots, n \ \ \text{and} \ \ k - \sum_{i=1}^n k_i \ge 0. $$ 

The result follows from taking the symplectic area of $\eqref{eq:M2k}$, which is 

$$ ks + \sum_{i=1}^n k_i (1 - 2s) = s(k - \sum_{i=1}^n k_i) + (1 - s)(\sum_{i=1}^n k_i)
$$

\end{proof}

As pointed out before the above Proposition allows us to informally argue why
Conjecture \ref{cnj:toridistinct} should hold. Indeed, for $s > 1/2$, the number
of Maslov index $2$ holomorphic disks with boundary in $\Ts$ and with minimal
area $a = 1-s$ is $n^2$, by Proposition \ref{prp: Poten}. Hence the number of
Maslov index $2$ disks with boundary in $\Theta^{k_1}_{s_1} \times \cdots \times
\Theta^{k_l}_{s_l}$ and with minimal area is at most $\sum_{i=1}^l k_i^2 <
(\sum_{i=1}^l k_i)^2 = n^2$, for $l>1$.

\section{Proof of Theorem  \ref{thm: main} - Bulk deformations} \label{sec: Proof}

In this section we use bulk deformations to prove that the tori $\Ts$ are
non-displaceable for $n$ even and $s \in [1/2, 1)$, as done in \cite{FO312}
for the case $n = 2$. In \cite{FO312}, Fukaya-Oh-Ohta-Ono used the cocycle
Poincar\'{e} dual to the anti-diagonal in $\CP^1 \times \CP^1$ to bulk-deform
Floer-homology. In this section we will bulk-deform Floer-homology by an element
of the form $T^{\rho}[h] \in H^*(\PN, \Lambda_+)$, where $[h] \in H^2(\PN, \Z)$. 

For $1 \le i \le n$, let $h_i$ be the cocycle Poincar\'e dual to
$\{ y_i = 0\} \subset (\CP^1)^n$.

\begin{prp} \label{prp: Bulk def Poten}
 
The potential for the Lagrangian tori $\Theta^n_{s}$, bulk deformed by the 
cocycle $\bb = T^{\rho}[(k_1 + k_n)h_1 + \cdots + (k_{n-1} + k_n) h_{n-1} + k_n h_n ] \in 
C^2(\PN, \Lambda_+)$ is given by 
 
\[ \PO^{\Ts}_{\bb}(b) = u + \frac{T}{u}\left(1 + w_1 + \cdots + w_{n-1}\right)\left(1 +
\frac{e^{k_1T^{ \rho}}}{w_1} + \cdots + \frac{e^{k_{n-1}T^{
\rho}}}{w_{n-1}}\right)e^{k_{n}T^{ \rho}} \] 
 
\end{prp}

\begin{proof}
  The relative classes $\beta$, $\alpha_j$ have no intersection with 
  $\{ y_k = 0\}$ viewed as a cycle in $\PN \setminus \Ts$. Therefore
  the disk in the class $H_i - \beta - \alpha_i + \alpha_j$ intersect
  $\{ y_k = 0\}$ if and only if $k=i$, and with multiplicity 1. Hence,
  the coefficient of the monomial $\sfrac{Tw_j}{uw_i}$ is bulk-deformed
  by $\bb_s$ to $e^{(k_i + k_n)T^{\rho}}$.
  
\end{proof}

\begin{lem} \label{lem: Critical Pts}
  
The potential for the Lagrangian tori $\Theta^n_{s}$, bulk deformed by the 
cocycle $\bb = T^{\rho}[(k_1 + k_n)h_1 + \cdots + (k_{n-1} + k_n) h_{n-1} + k_n h_n ] \in 
C^2(\PN, \Lambda_+)$ have its critical points given by:

\[ w_i = \epsilon_i e^{\frac{k_i}{2}T^\rho}, \: \: u = \epsilon_n e^{\frac{k_n}{2}T^\rho} T^{\frac{1}{2}} (1 + \sum_{i \ge
1}^{n-1} \epsilon_i e^{\frac{k_i}{2}T^\rho}), \]

where $\epsilon_i = \pm 1$.

\end{lem}

\begin{proof}
  
For easier notation, let $b_i = e^{k_iT^{\rho}}$. Taking the differential of the
bulk deformed potential $\PO^{\Ts}_{\bb}(b) $ with respect to $w_i$ and equating
to $0$, we get, after multiplying by $w_i$, equations 

\begin{gather} \label{eq: i} (i): w_i + \sum_{j \neq i} \frac{b_jw_i}{w_j} -
b_i(\frac{1}{w_i} + \sum_{j \neq i} \frac{w_j}{w_i}) = 0. \end{gather}

Summing all the equations $(1), \dots, (n)$, we end up with

\begin{equation*} 
  \sum_{i=1}^{n-1} w_i - \sum_{i=1}^{n-1} \frac{b_i}{w_i} = 0 
\end{equation*}



Let \[L = \sum_{i=1}^{n-1} w_i = \sum_{i=1}^{n-1} \frac{b_i}{w_i}.\] 

We have that 

\begin{gather*} w_iL - b_i = \sum_{j \neq i} \frac{b_jw_i}{w_j},\\
\frac{L}{w_i} - 1 = \sum_{j \neq i} \frac{w_j}{w_i}.
\end{gather*}

Substituting the above into equations $(i)$ (see \eqref{eq: i}), we get that

\begin{equation} \label{eq: 1 + L} 
 \left(w_i - \frac{b_i}{w_i}\right)\left(1 + L \right) = 0  
\end{equation}

So if $u, w_1, \dots w_{n-1}$ are critical points of the
bulk deformed potential $\PO^{\Ts}_{\bb}(b)$, besides equation \eqref{eq: 1 + L}, we
must have

\begin{equation}\label{eq: del_u}
 \del_u \PO^{\Ts}_{\bb} = 1 - \frac{b_nT}{u^2}(1 + L)^2 = 0  
\end{equation}

Hence $L \neq -1$, and therefore

\[ w_i = \sqrt{b_i} = \epsilon_i e^{\frac{k_i}{2}T^\rho}, \: \: u = \sqrt{b_n}T^{\frac{1}{2}}(1 + L) =
\epsilon_n e^{\frac{k_n}{2}T^\rho}T^{\frac{1}{2}} (1 + \sum_{i \ge
1}^{n-1} \epsilon_i e^{\frac{k_i}{2}T^\rho}), \]
    
\end{proof}

We call the \emph{valuation} of an element in $\Lambda_+$ the smallest exponent
with non-zero coefficient. Looking at the expression of the critical points of
the previous Lemma, one can see that:

\begin{lem} \label{lem: val=1/2}
 Looking at the critical points given on Lemma \ref{lem: Critical Pts} we have that,
 the valuation of $u$ is not $1/2$ if and only if $n = 2m$ and $m-1$ $\epsilon_i$'s are
 equal to $1$ while the other $m$ $\epsilon_i$'s are equal to $-1$, where $i = 1, \dots, 2m-1$.
 In that case, the valuation of $u$ is $T^{1/2 + \rho}$, provided $\sum_{i=1}^{2m-1} \epsilon_i k_i \ne 0$. 
\end{lem}

Now we recall that 

$$u = z_\beta = T^s \exp(b \cap \del \beta)$$

for the class $\beta$ defined in the beginning of Section \ref{subsec: PotTs}. 
By Lemma \ref{lem: val=1/2}, we have:

\begin{cor} \label{cor: CritcPotTsm} Take $s> 1/2$ and consider the cocycle
$\bb_s = T^{s - 1/2}[(k_1 + k_{2m})h_1 + \cdots + (k_{2m-1} + k_{2m}) h_{2m-1} +
k_{2m} h_{2m} ] \in C^2(\Pm, \Lambda_+)$. Assume that not all $k_i$'s are $0$,
for $i= 1, \dots, 2m-1$, i.e., $[\bb_s]$ is not a multiple of the monotone
symplectic form. Then there exists $b_s$ a critical point of
$\PO^{\Tsm}_{\bb_s}$. 
\end{cor}

Recalling that $\Tsm$ satisfy Assumption \ref{ass: ass} (Propositions \ref{prp: Ass}), for some almost complex structure $J$, and noting that
$\Tsm$ is a contractible Lagrangian torus of $\Pm$, we have that $(\Pm,\Tsm)$
satisfy all the hypothesis of Corollary \ref{cor: PotFloerHom}. Therefore, from
Corollaries \ref{cor: PotFloerHom} and \ref{cor: CritcPotTsm}, we deduce:

\begin{thm} \label{thm: FloerHomTsm}
  For $s \ge 1/2$ there exists a bulk $[\bb_s] \in H^2(\Pm, \Lambda_+)$ and a weak 
  bounding cochain $b_s \in H^1(\Tsm,\Lambda_0)$ such that $HF(\Tsm, 
  (b_s,\bb_s);\Lambda_{0,nov}) \cong H(\Tsm,\Lambda_{0,nov})$.
\end{thm}

This proves the first part of Theorem \ref{thm: heavy}.
Theorem \ref{thm: main} follows from Theorem \ref{thm: FOOOnonDisp} and 
Theorem \ref{thm: FloerHomTsm}. $\qed$ 

Corollary \ref{cor: ProductTori} follows from the same arguments as above using 
that

$$ \PO^{\Theta^{k_1}_{s_1} \times \cdots \times \Theta^{k_l}_{s_l} \times
(S^2_{\OP{eq}})^{n - \sum_i k_i}}_{ \bb} = \PO^{\Theta^{k_1}_{s_1}}_{\bb} + \cdots + \PO^{\Theta^{k_l}_{s_l}}_{ \bb} 
+ \PO^{(S^2_{\OP{eq}})^{n - \sum_i k_i}}_{\bb} \qed$$



\section{Quasi-morphisms and quasi-states} \label{sec: Quasimorph}

In this section we prove the last part of Theorem \ref{thm: heavy}. It follows 
arguments similar to \cite[Theorem~23.4]{FO311b}. 

\begin{lem} \label{lem: QuantCohom}
  For any $\bb = T^{\rho}[l_1 h_1 + \cdots + l_{n-1} h_{n-1} + l_n h_n ] \in 
C^2(\PN, \Lambda_+)$, the bulk deformed Quantum cohomology \cite[Section~5]{FO311b} is semi-simple. 
\end{lem}

\begin{proof} By \cite[Theorem~1.1.1]{FO316} (see also
\cite[Theorem~6.1]{FO310}, for the Fano case) we have an isomorphism between the
bulk deformed Quantum cohomology of a toric symplectic manifold and the Jacobian
Ring of the bulk deformed toric potential. If the bulk deformed toric potential
has only non-degenerate critical points, we can split the Quantum cohomology
ring into orthogonal algebra summands according to the factors corresponding to
the critical points under the isomorphism given in \cite[Theorem~1.1.1]{FO316}.

Naming now $z_i = z_{\beta_i}$ \eqref{def: coord z}, for $\beta_i$ the class of 
Maslov index 2 holomorphic disk intersecting $\{x_i = 0\}$, we have that the 
bulk deformed potential of a toric fiber is:

\begin{equation} \label{eq: toricBulkPot}
  \PO_\bb = z_1 + \cdots + z_n + \frac{Te^{l_1T^\rho}}{z_1} + \cdots + 
  \frac{Te^{l_nT^\rho}}{z_n},
\end{equation}
whose critical points are given by $(z_1, \dots, z_n) = (\epsilon_1 T^{1/2} 
e^{l_1T^\rho/2}, \dots , \epsilon_n T^{1/2} e^{l_nT^\rho/2})$. Hence, there
are $2^n$ idempotents of $QH_\bb(\PN;\Lambda_{0,nov})$, $\be_1^\bb$, \dots, 
$\be_{2^n}^\bb$ for which

$$QH_\bb(\PN;\Lambda_{0,nov}) = \bigoplus_{i=1}^{2^n} \Lambda_{0,nov} \be_i^\bb  .$$
  
\end{proof}

In \cite[Section~17, (17.18)]{FO311b}, given $X$ a symplectic manifold
and $L$ a relatively spin Lagrangian submanifold, Fukaya-Oh-Ohta-Ono construct an homomorphism:

\begin{equation} \label{eq: iqmHom}
  i^*_{\qm,(b,\bb)}: QH_\bb(X;\Lambda_{0,nov}) \to HF(L, (b,\bb); \Lambda_{0,nov}),
\end{equation}
which is proven to be a ring homomorphism in \cite{AFO3Prep1}, see 
\cite[Remark~17.16]{FO311b} and \cite[Section~4.7]{FO310c}.

Applying Lemma \ref{lem: QuantCohom} for $\Pm$ and $\bb_s$ given in Theorem
\ref{thm: FloerHomTsm}, using that $i^*_{\qm,(b_s,\bb_s)}$ is unital and $HF(\Tsm,(b_s,\bb_s); \Lambda_{0,nov}) 
\ne 0$, we have:

\begin{prp} \label{prp: idempotent}
  There exists an idempotent $\be_s \in QH_{\bb_s}(\Pm;\Lambda_{0,nov})$ for which $i^*_{\qm,(b_s,\bb_s)}(\be_s) \ne 0$
  in $HF(\Tsm,(b_s,\bb_s); \Lambda_{0,nov})$.
\end{prp}

Theorem \ref{thm: heavy} follows then from Proposition \ref{prp: idempotent} and 
Theorem 18.8 of \cite{FO311b}. \qed

\section{Tori in $\BlIII$} \label{sec: Bl3}

In this section we prove Theorem \ref{thm: Bl3heavy}. We will describe a model
for $(\BlIII, \omega_\epsilon) = (\PxPBlII, \omega_\epsilon)$ which is
equivalent to performing two blowups of capacities $\epsilon$ centred at the
rank 0 elliptic singularities (corners) of the singular fibration of $\PxP$
described in \cite{FO312}, see Figure \ref{fig: Bl3}.

\begin{figure}[h!]   
  
\begin{center}

\centerline{\includegraphics[scale=0.4]{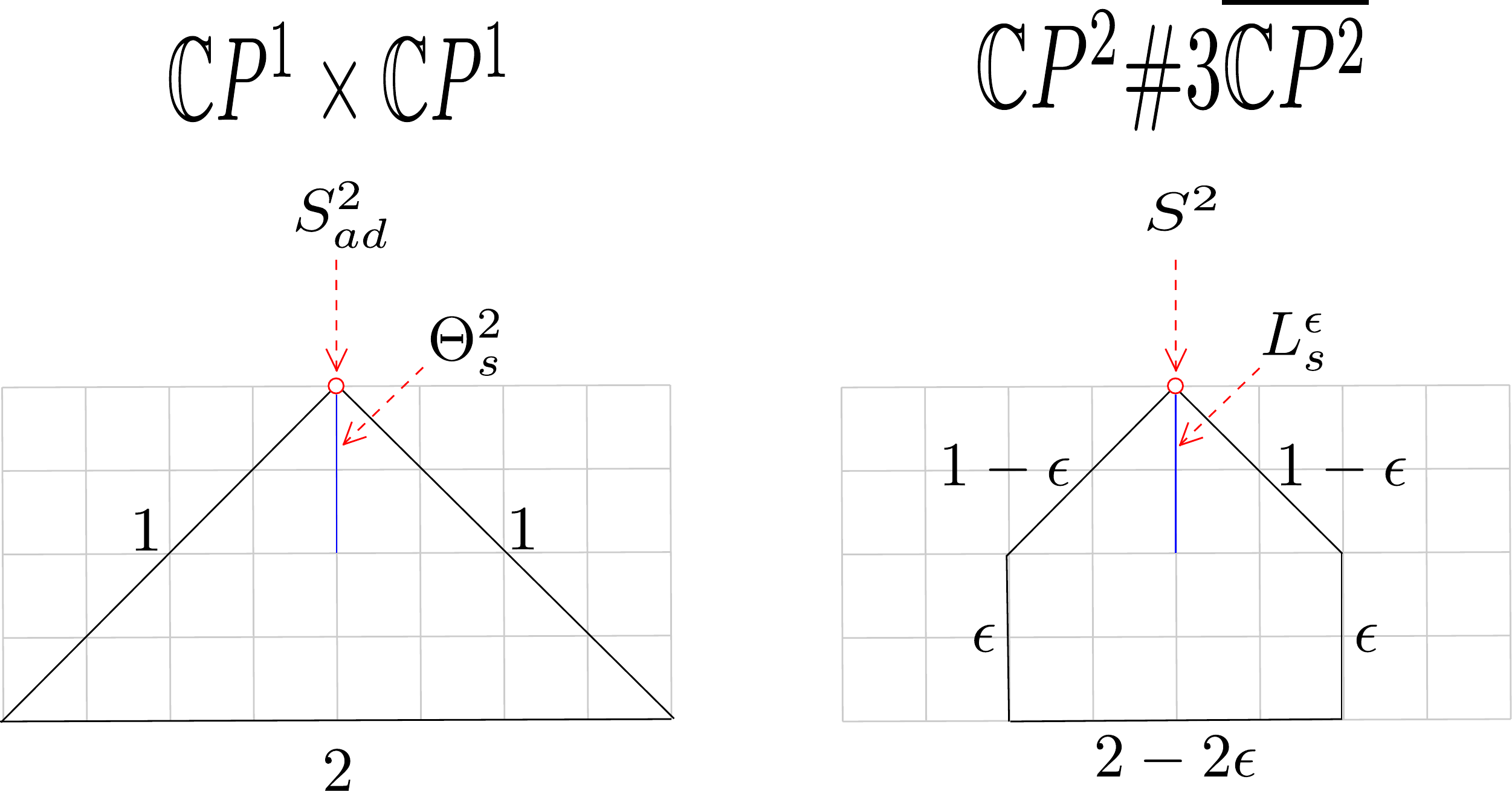}}

\caption{Singular fibrations of $\PxP$ and $\BlIII$.}
\label{fig: Bl3} 

\end{center} 
\end{figure}

Consider $\PxP$ with coordinates $([x_1:y_1], [x_2:y_2])$ as in Section
\ref{subsec: DfnTs}. Consider also the tori $\Theta^2_s$, the function $f =
x_1x_2/y_1y_2$, the relative class $\beta$ and $\alpha := \alpha_1$ and the
divisor $D = f^{-1}(1)\cup \{y_1= 0\} \cup \{y_2 = 0\}$, as defined in Section
\ref{subsec: PotTs}. 

From Proposition \ref{prp: SpecialLag} and \cite[Lemma~3.1]{Au07},
we have that $2[D] \in H_2(\PxP \setminus \Theta^2_s)$ is Poincar\'e dual to the
Maslov class $\mu_{\Theta^2_s} \in H^2(\PxP,\Theta^2_s)$. In particular the
Maslov index 2 holomorphic disks, computed in Proposition \ref{prp: Poten} for $n=2$,
do not intersect $\overline{f^{-1}(1)} \cap \{y_1= 0\} = ([1:0],[0:1]) = p_1$ and
$\overline{f^{-1}(1)} \cap \{y_2= 0\} = ([0:1],[1:0]) = p_2$.  

Let $B_i(\epsilon)$ be the ball of capacity \cite[Section~12]{MDSaBook_SympTop}
$\epsilon$ (radius $\sqrt{\epsilon / \pi}$) centered at $p_i$, in the coordinate
plane $x_i = 1$, $y_j = 1$, $i,j = 1 ,2$, $i \ne j$. Denote $S_i(\epsilon) =
\del B_i(\epsilon)$. Let $(\BlIII, \omega_\epsilon)$ be the result of blowing up
\cite[Section~7]{MDSaBook_SympTop} $\PxP$ with respect to $B_1(\epsilon)$ and
$B_2(\epsilon)$, so that the exceptional curves $E_i$ (coming from collapsing
the Hopf fibration in $S_i(\epsilon)$) have symplectic area $\omega_\epsilon
(E_i) = \epsilon$, $i = 1,2$. Let $j_\epsilon$ be the induced complex structure
and $L_s^\epsilon$ correspond to $\Theta_s^2$ after the blowup. Note that
$\epsilon$ can take any value in $(0,1)$, so that $B_1(\epsilon) \cap
B_2(\epsilon) = \emptyset$.

Note also that $f = x_1x_2/y_1y_2$ is constant along the fibers of the Hopf
fibration of both $S_1(\epsilon)$ and $S_2(\epsilon)$. In particular it give
rise to a $(j_\epsilon,j)$-holomorphic function $\tilde{f}: \BlIII \to \CP^1$. 

For computing the potential for $L_s^\epsilon$ it is interesting that the disks
of Proposition \ref{prp: Poten}, remain essentially the same. This can be
obtained by stretching the complex structure $j_\epsilon$. So take $\delta$ small
enough so that $B_1(\delta) \cup B_2(\delta)$ does not intersect any Maslov
index 2 holomorphic disk. Consider a diffeomorphism $\varphi: (\BlIII,
\omega_\epsilon) \to (\BlIII, \omega_\delta)$ coming from a finite neck stretch
\cite{EliGiHo10,CompSFT03} along $S_i(\epsilon + \delta') \subset (\BlIII,
\omega_\epsilon)$ \cite{CompSFT03,EliGiHo10}, see also \cite[Section~3]{Vi14},
which sends $L_s^\epsilon$ to $L_s^\delta$. The diffeomorphism $\varphi$ is
equivalent to considering an inflation along the exceptional curves $E_i$, $i = 
1,2$. Set $J_\delta = \varphi^*j_\delta$, an $\omega_\epsilon$ compatible 
almost complex structure.

\begin{lem} \label{lem: potBl3}
  We have that $(\BlIII,L_s^\epsilon,J_\delta)$ satisfy Assumption \ref{ass: ass}. 
  The potential function for $L_s^\epsilon$ with respect to $J_\delta$, is given by:
  
  \begin{equation}
    \PO^{L_s^\epsilon} = u + \frac{T}{u}(1 + w)(1 + \frac{1}{w}) + T^{1 - 
    \epsilon}(w + \frac{1}{w})
  \end{equation}

\end{lem}

\begin{proof}
  It is enough to compute the $j_\delta$-holomorphic disks with boundary in 
  $L_s^\delta$. The $j_\delta$-holomorphic disks that don't intersect the 
  exceptional divisors $E_1$, $E_2$, corresponds to the holomorphic disks in 
  $\PxP$ with boundary in $\Theta^2_s$, which gives the terms 
  $$ u + \frac{T}{u}(1 + w)(1 + \frac{1}{w})$$ of $\PO^{L_s^\delta}$, and are regular.
 
 Let $\tilde{D}$ be the proper transform of the divisor $D \in \PxP$. It can be
 checked that, twice $\tilde{D} + E_1 + E_2$ is Poincar\'e dual to the Maslov
 class $\mu_{L_s^\delta}$. This implies Assumption \ref{ass: A1}, as in the
 proof of Proposition \ref{prp: Ass}. Moreover, Maslov index 2 disks intersects
 $\tilde{D} + E_1 + E_2$ once. Which means that if a $j_\delta$-holomorphic disk
 $u$ intersects either $E_1$ or $E_2$, by positivity of intersection, it does
 not intersect $\tilde{D}$ and hence $\tilde{f}\circ u: \D \to \C^*$ must be
 constant. There are two Maslov index 2 disks in the fiber $ \tilde{f}^{-1}(c)$,
 for $c \in \gamma_s$. Looking at the intersections with $E_i$, and the proper
 transform of $\{x_i = 0\}$ and $\{y_i = 0 \}$, we can see that the relative classes
 of these disks are $H_1 - E_1 + \alpha$ and $H_2 - E_2 - \alpha$ (for some
 orientation of $\alpha$). Since, $\omega_\epsilon (H_i - E_i \pm \alpha) = 1 -
 \epsilon$, we get the remaining term $$T^{1 - \epsilon}(w + \frac{1}{w}).$$ 
 
 To show regularity of the above disks, one notes that the pre-image under
 $\tilde{f}$ of a small neighbourhood $\mathcal{N}_s$ of $\gamma_s$ contain the
 whole family of the above disks and is actually
 toric. Moreover, $(\tilde{f}^{-1}(\mathcal{N}_s), L_s^\delta)$ is $T^2$-homogeneous
 \cite{EL15b}, or if you will, $S^1$-pseudohomogeneous (Definition \ref{dfn:
 almHomogeneuos}) for a $j_\delta$-holomorphic $S^1$-action transverse to $\del
 \alpha$, which shows Assumption \ref{ass: A2}. 
 
 The choice of spin structure is given by trivialising $TL_s^\epsilon$ according
 to $\{\alpha, \beta\}$ and is so that the evaluation map is orientation
 preserving, as in the proof of Proposition \ref{prp: Poten}. See also
 \cite[Section~5.5]{Vi13} and \cite[Section~8]{Cho04}.
  
\end{proof}

\begin{rmk} The above potential can also be computed by a technique similar to
the one developed in \cite{FO312} and also by some gluing procedure similar to
the one developed in Section 5.2 of the ArXiv.1002.1660v1 version of
\cite{FO312} and in \cite{Wu15}. 
\end{rmk}

\begin{rmk} For each $\delta' > 0$, the family $\{L_s^\epsilon: s \in [1/2, 1 -
\delta'] \}$ can be seen as fibres of an almost toric fibration (ATF) of
$\BlIII$, represented by an almost toric base diagram (ATBD) analogous to the
one in Figure 9 $(A_3)$ of \cite{Vi16a}. In fact, the singular fibration
described by the second diagram in Figure \ref{fig: Bl3} can be thought as a
limit of ATFs described by sliding nodes of the ATBD in Figure 9 $(A_3)$ of \cite{Vi16a}.
Moreover, the potential $\PO^{L_s^\epsilon}$ can be obtained
from the toric potential $$ \PO^{\mathrm{toric}} = u_1 + u_2 + \frac{T}{u_1} +
\frac{T}{u_2} + \frac{T^{1 - \epsilon}u_1}{u_2} + \frac{T^{1 -
\epsilon}u_2}{u_1},$$ via wall-crossing transformation $u = u_1(1 + w)$, $w =
u_2/u_1$, giving another example where actual computations meet wall-crossing predictions
\cite{Au07,Au09,Vi13}. 
 
\end{rmk}

Let $\fs \in C^2(\BlIII)$ be the cocycle Poincar\'e dual to $\{y_1 = 0\} \cup E_1$, 
so $[\fs] = H_1 - E_2 + E_1$. Analogous to Proposition \ref{prp: Bulk def Poten}, we
have:

\begin{prp} \label{prp: bulkdef potBl3}
  The potential for $L_s^\epsilon$, bulk deformed by the cocycle $\bb = T^{\rho}\fs
  \in C^2(\BlIII, \Lambda_+)$ is given by:  
  
  \begin{equation} \label{eq: PotBulked Bl3}
    \PO^{L_s^\epsilon}_{\bb} = u + \frac{T}{u}(1 + w)(e^{T^\rho} + \frac{1}{w}) + T^{1 - 
    \epsilon}(e^{T^\rho}w + \frac{1}{w}).
  \end{equation}
  
\end{prp}

We can then compute the critical points of $\PO^{L_s^\epsilon}_{\bb}$ and obtain:

\begin{lem} \label{lem: crit PotBl3}
  We have that $w = - e^{\frac{-T^\rho}{2}}$ and $u = \pm T^{\frac{1}{2}}(1 
  - e^{\frac{-T^\rho}{2}})^{\frac{1}{2}}(e^{T^\rho - 
  e^{\frac{T^\rho}{2}}})^{\frac{1}{2}}$ are critical points of 
  $\PO^{L_s^\epsilon}_{\bb}$. The valuations of $w$ and $u$ are respectively
  $0$ and $1/2 + \rho$.
  
\end{lem}

Since we have that $\int_{\beta}\omega_\epsilon = s$ and $\int_{\alpha} \omega_\epsilon 
= 0$: 

\begin{lem} \label{lem: crit2 PotBl3} For $s > 1/2$ and $\bb_s^\epsilon = T^{s -
1/2}[\fs]$, there exists a weak bounding cochain $b_s^\epsilon \in
H^1(L_s^\epsilon, \Lambda_0)$ which is a critical point of
$\PO^{L_s^\epsilon}_{\bb_s^\epsilon}$. 
\end{lem}

Following similar arguments as in Sections \ref{sec: Proof} and \ref{sec: Quasimorph}, we 
are able to prove Theorem \ref{thm: Bl3heavy} and consequently Theorem \ref{thm: 
Bl3}. \qed

\bibliographystyle{abbrv}
\bibliography{SympRefs}

\end{document}